\documentclass[reqno,11pt]{amsart}
\usepackage{amsmath,amsfonts,amscd,amssymb,enumitem,
extarrows,latexsym}

\usepackage{hyperref}
\usepackage{soul}
\usepackage{mathrsfs}  

\usepackage[nobysame]{amsrefs}
\usepackage{comment}
\usepackage{tikz}
\usepackage{marginnote}
\usepackage{anysize}
\usepackage{bbm}
\usepackage{cancel}
\usepackage[normalem]{ulem}

\numberwithin{equation}{section}

\newtheorem{theorem}{Theorem}[section]
\newtheorem{lemma}[theorem]{Lemma}
\newtheorem{proposition}[theorem]{Proposition}
\newtheorem{corollary}[theorem]{Corollary}

\theoremstyle{definition}
\newtheorem{definition}[theorem]{Definition}
\newtheorem{example}[theorem]{Example}

\theoremstyle{remark}
\newtheorem{remark}[theorem]{Remark}

\numberwithin{equation}{section}

\newcommand{\ind}{\operatorname{ind}}
\newcommand{\id}{\operatorname{id}}
\newcommand{\dom}{\operatorname{dom}}

\newcommand{\cc}{{cc}}
\newcommand{\dm}{\partial M}
\newcommand{\CC}{\mathbb{C}}
\newcommand{\CCC}{\mathcal{C}}
\newcommand{\RR}{\mathbb{R}}
\newcommand{\cRR}{\mathcal{R}}
\newcommand{\ZZ}{\mathbb{Z}}

\newcommand{\AAA}{\mathcal{A}}

\newcommand{\DD}{\mathcal{D}}
\newcommand{\EE}{\mathcal{E}}

\newcommand{\cHH}{\check{H}}
\newcommand{\hHH}{\hat{H}}
\newcommand{\cPi}{\check{\Pi}}
\newcommand{\cPP}{\check{P}}
\newcommand{\cTT}{\check{T}}

\newcommand{\bfu}{\mathbf{u}}
\newcommand{\bfv}{\mathbf{v}}
\newcommand{\upper}{\uppercase\expandafter}

\newcommand{\n}{\nabla}
\newcommand{\p}{\partial}

\newcommand{\codim}{\operatorname{codim}}
\newcommand{\Hom}{\operatorname{Hom}}
\newcommand{\End}{\operatorname{End}}
\newcommand{\range}{\operatorname{range}}
\newcommand{\ad}{{\rm ad}}
\newcommand{\coker}{\operatorname{coker}}

\begin{document}

\title[Cauchy data spaces and APS index]{Cauchy data spaces and Atiyah--Patodi--Singer index on non-compact manifolds}


\author{Pengshuai Shi}
\address{Department of Mathematics,
Northeastern University,
Boston, MA 02115,
USA}
\curraddr{}
\email{shi.pe@husky.neu.edu}
\thanks{}


\date{}

\dedicatory{}

\begin{abstract}
We study the Cauchy data spaces of the strongly Callias-type operators using maximal domain on manifolds with non-compact boundary, with the aim of understanding the Atiyah--Patodi--Singer index and elliptic boundary value problems.
\end{abstract}

\maketitle


\section{Introduction}\label{S:intro}

In our previous papers \cites{BrShi17,BrShi17-2} with Braverman, we studied the boundary value problems of strongly Callias-type operators on manifolds with non-compact boundary. In particular, for the Atiyah--Patodi--Singer (or APS) boundary value problem, we found a formula to compute the APS index. An interesting term in the formula is a boundary invariant on a model manifold which behaves like the difference of two individual eta-invariants. We call it \emph{relative eta-invariant}. One question that remains to be answered is a spectral interpretation of this invariant.

Another notion involved in the study of boundary value problems is the space of Cauchy data. In particular, the APS index (on manifold with compact boundary) can be computed in terms of the projections onto Cauchy data spaces, which provides another way of understanding the eta invariant. In this paper, we address the APS index for strongly Callias-type operators from this perspective. Traditionally, Cauchy data spaces of Dirac-type operators can be built through the $L^2$-closure of boundary restrictions of smooth solutions on partitioned (compact) manifolds. This approach involves pseudo-differential calculus, i.e., a Cauchy data space is the range of the $L^2$-extension of Caldr\'on projector. (cf. \cites{BW93,Seeley}.)

A different but more general approach is established on the maximal domain of an operator on a manifold with boundary by Booss-Bavnbek and Furutani \cite{BF98}. When the operator is symmetric, there is a symplectic structure on the space of boundary values of sections in maximal domain. The (maximal) Cauchy data space is a subspace of this boundary value space. And under natural assumptions, such a Cauchy data space gives rise to Fredholm-Lagrangian property. A good feature of this treatment is that it gets rid of pseudo-differential calculus. We refer the reader to \cite{BFW} for a nice exposition on these two approaches.

We shall adopt the maximal domain approach to study the Cauchy data spaces of strongly Callias-type operators on manifolds with non-compact boundary. Since we mainly consider the graded operator, we will care more about the Fredholmness than the Lagrangian. We give formulas of the APS index through the APS projection and projections onto Cauchy data spaces (Theorems \ref{T:indid} and \ref{T:L2indid}). We also prove the twisted orthogonality of Cauchy data spaces (Theorem \ref{T:L2decomp}). These results can be compared with the results in \cites{BW93,Woj}. At last, we interpret certain Cauchy data spaces as elliptic boundary conditions in the sense of \cite{BrShi17} (Theorem \ref{T:Cauchyellbc}). In \cite{BBC}, Ballmann, Br\"uning and Carron discussed the Cauchy data spaces on a semi-infinite cylinder model. Since the growth of the potential in our operator controls the behavior at infinity, we do not need to consider extended solutions. (Compare Theorem \ref{T:Cauchyellbc} with \cite{BBC}*{Theorem C}.)

\section{Preliminaries}\label{S:preli}

In this section, we give a short review about the boundary value problems of strongly Callias-type operators. All the contents except Subsection \ref{SS:UCP} can be found in \cites{BrShi17,BrShi17-2} which generalize some results of B\"ar and Ballmann \cite{BB12} to manifolds with non-compact boundary.

\subsection{Strongly Callias-type operators}\label{SS:Callias}

Let $M$ be a complete Riemannian manifold (possibly with boundary) and let $E\to M$ be a Dirac bundle over $M$, \cite{LM89}*{Definition~II.5.2}. In particular, $E$ is a Hermitian bundle endowed with a Clifford multiplication $c:T^*M\to \End(E)$ and a compatible Hermitian connection $\n^E$. Suppose that $E=E^+\oplus E^-$ is $\ZZ_2$-graded such that the Clifford multiplication $c(\xi)$ is odd and the Clifford connection is even with respect to this grading. Then one can form the $\ZZ_2$-graded Dirac operator
\[
D\;:=\;\left(
\begin{matrix}
0 & D^- \\
D^+ & 0 
\end{matrix}
\right),
\]
where
\(
D^\pm\;:\;C^\infty(M,E^\pm)\to C^\infty(M,E^\mp)
\)
are formally adjoint to each other. 

Let $\Psi\in{\rm End}(E)$ be a self-adjoint bundle map (called a \emph{Callias potential}) which is odd-graded, i.e.
\[
\Psi\;=\;\left(
\begin{matrix}
0 & \Psi^- \\
\Psi^+ & 0 
\end{matrix}
\right),
\]
where $\Psi^\pm\in\Hom(E^\pm,E^\mp)$ are adjoint to each other. Then we have a formally self-adjoint Dirac-type operator on $E$
\begin{equation}\label{E:Callias}
	\DD\;:=\;D\,+\,\Psi\;=\;\left(
	\begin{matrix}
	0 & D^-+\Psi^- \\
	D^++\Psi^+ & 0
	\end{matrix}
	\right)\;=:\;\left(
	\begin{matrix}
	0 & \DD^- \\
	\DD^+ & 0
	\end{matrix}
	\right).
\end{equation}
Note that
\[
	\DD^2\;=\;D^2+\Psi^2+[D,\Psi]_+,
\]
where $[D,\Psi]_+:=D\circ\Psi+\Psi\circ D$ is the anticommutator of the operators $D$ and $\Psi$.

\begin{definition}\label{D:saCallias}
We call $\DD$ (or $\DD^+,\DD^-$) a \emph{strongly Callias-type operator} if
\begin{enumerate}
\item $[D,\Psi]_+$ is a zeroth order differential operator, i.e. a bundle map;
\item for any $R>0$, there exists a compact subset $K_R\subset M$ such that
\[
	\Psi^2(x)\;-\;\big|[D,\Psi]_+(x)\big|\;\ge\;R
\]
for all $x\in M\setminus K_R$.
\end{enumerate}
\end{definition}

Assume the Riemannian metric and the Dirac bundle $E$ both have product structure in a tubular neighborhood $U\subset M$ of the boundary. Let $t$ be the inward-pointing normal coordinate near the boundary so that the inward unit normal vector to the boundary is given by $\tau=dt$. Then near the boundary, a Callias-type operator $\DD$ takes the form
\begin{equation}\label{E:productDD}
	\DD\; = \;c(\tau)(\p_t+\hat\AAA)\;=\;
	\left(
	\begin{matrix}
	0 & c(\tau) \\
	c(\tau) & 0
	\end{matrix}
	\right)
	\left(
	\begin{matrix}
	\p_t+\AAA & 0 \\
	0 & \p_t+\AAA^\sharp
	\end{matrix}
	\right), 
\end{equation}
where $\AAA:C^\infty(\dm,E^+|_{\dm})\to C^\infty(\dm,E^+|_{\dm})$ and $\AAA^\sharp:C^\infty(\dm,E^-|_{\dm})\to C^\infty(\dm,E^-|_{\dm})$ are formally self-adjoint operators satisfying
\begin{equation}\label{E:anticommu}
\AAA^\sharp\;=\;c(\tau)\circ\AAA\circ c(\tau).
\end{equation}
$\AAA$ and $\AAA^\sharp$ are also (non-graded) strongly Callias-type operators. In particular, they have discrete spectrum. We call $\AAA$ (resp. $\AAA^\sharp$) the \emph{restriction of $\DD^+$ (resp. $\DD^-$) to the boundary}.

\subsection{Minimal and maximal extensions}\label{SS:minmaxext}

For a Dirac bundle $E$ over $M$, we set $C_c^\infty(M,E)$ to be the space of smooth sections of $E$ with compact support and $C_\cc^\infty(M,E)$ to be the space of smooth sections of $E$ with compact support in $M\setminus\p M$. We denote by $L^2(M,E)$ the Hilbert space of square-integrable sections of $E$, which is the completion of $C_c^\infty(M,E)$ with respect to the norm induced by the $L^2$-inner product
\[
(u_1;u_2)_{L^2(M)}\;:=\;\int_M\,\langle u_1;u_2\rangle\, dV,
\]
where $\langle\cdot;\cdot\rangle$ denotes the fiberwise inner product and $dV$ is the volume form on $M$. Similar spaces can be defined on the boundary $\dm$. We usually use letters $u,v,\cdots$ to denote sections on $M$ and use bold letters $\bfu,\bfv,\cdots$ to denote sections on $\dm$.

Let $\DD^+$ be a strongly Callias-type operator. We denote $\DD_\cc^+:=\DD^+|_{C_\cc^\infty(M,E^+)}$ and view it as an unbounded operator from $L^2(M,E^+)$ to $L^2(M,E^-)$. The \emph{minimal extension} $\DD_{\min}^+$ of $\DD^+$ is the operator whose graph is the closure of that of $\DD_\cc^+$. The \emph{maximal extension} $\DD_{\max}^+$ of $\DD^+$ is defined to be $\DD_{\max}^+=\big((\DD^-)_{cc}\big)^\ad$, where the superscript ``$\ad$'' denotes the adjoint of the operator in the sense of functional analysis. Both $\DD_{\min}^+$ and $\DD_{\max}^+$ are closed operators. Their domains, $\dom\DD_{\min}^+$ and $\dom\DD_{\max}^+$, become Hilbert spaces equipped with the \emph{graph norm} $\|\cdot\|_{\DD^+}$, which is the norm associated with the inner product
\[
	(u_1;u_2)_{\DD^+}\;:=\;
	(u_1;u_2)_{L^2(M)}\;+\;(\DD^+u_1;\DD^+u_2)_{L^2(M)}
\]

\subsection{Sobolev spaces on the boundary}\label{SS:bdrySob}

Since the boundary in our problem is in general non-compact, there is not a canonical way of defining Sobolev spaces on it. Naturally, we use the operator restricted to the boundary to define them.

\begin{definition}[\cite{BrShi17}*{\S3}]\label{D:bdrySob}
Let $\{\lambda_j\}_{j\in\ZZ}$ be the set of eigenvalues of $\AAA$ and $\{\bfu_j\}_{j\in\ZZ}$ the corresponding unit eigenvectors, which form an orthonormal basis of $L^2(\dm,E^+|_{\dm})$. For any $s\in\RR$, we define the $s^{\text{th}}$-order Sobolev space to be
\[
	H_\AAA^s(\dm,E^+|_{\dm})\;:=\;
	\Big\{\,\bfu=\sum_ja_j\bfu_j:\,\,\sum_j|a_j|^2(1+\lambda_j^2)^s<+\infty\,\Big\}.
\]
It is a Hilbert space with norm given by $\sum_j|a_j|^2(1+\lambda_j^2)^s$.
\end{definition}

\begin{remark}\label{R:bdrySob}
$H_\AAA^0(\dm,E^+|_{\dm})=L^2(\dm,E^+|_{\dm})$. For any $s\in\RR$, there is a perfect pairing
\[
\begin{aligned}
H_\AAA^s(\dm,E^+|_{\dm})\times H_\AAA^{-s}(\dm,E^+|_{\dm})\ &\to\ \CC, \\
\Big(\sum_j a_j\bfu_j,\sum_j b_j\bfu_j\Big)\qquad\qquad&\mapsto\ \sum_j\overline{a}_jb_j.
\end{aligned}
\]
Therefore, $H_\AAA^s(\dm,E^+|_{\dm})$ and $H_\AAA^{-s}(\dm,E^+|_{\dm})$ are dual to each other.
\end{remark}

For $I\subset\RR$, let
\[
	P_I^\AAA\;:\;\sum_j a_j\bfu_j\;\mapsto\;\sum_{\lambda_j\in I}a_j\bfu_j
\]
be the spectral projection. Then for all $s\in\RR$,
\[
	H_I^s(\AAA)\;:=\;
	P_I^\AAA(H_\AAA^s(\dm,E^+|_{\dm}))\;\subset\;H_\AAA^s(\dm,E^+|_{\dm}).
\]

\begin{definition}[\cite{BrShi17}*{\S3}]\label{D:hybridSob}
For $a\in\RR$, we define the \emph{hybrid} Sobolev spaces
\[
\begin{aligned}
	\cHH(\AAA)&\;:=\;
	H_{(-\infty,a]}^{1/2}(\AAA)\;\oplus\;H_{(a,\infty)}^{-1/2}(\AAA), \\
	\hHH(\AAA)&\;:=\;
	H_{(-\infty,a]}^{-1/2}(\AAA)\;\oplus\;H_{(a,\infty)}^{1/2}(\AAA)
\end{aligned}
\]
with respective $\cHH$-norm, $\hHH$-norm
\[
\begin{aligned}
	\|\bfu\|_{\cHH(\AAA)}^2&\;:=\;
	\big\|P_{(-\infty,a]}^\AAA\bfu\big\|_{H_\AAA^{1/2}(\dm)}^2\;+\;
	\big\|P_{(a,\infty)}^\AAA\bfu\big\|_{H_\AAA^{-1/2}(\dm)}^2, \\
	\|\bfu\|_{\hHH(\AAA)}^2&\;:=\;
\|P_{(-\infty,a]}^\AAA\bfu\|_{H_\AAA^{-1/2}(\dm)}^2\;+\;
\|P_{(a,\infty)}^\AAA\bfu\|_{H_\AAA^{1/2}(\dm)}^2.
\end{aligned}
\]
\end{definition}

The spaces $\cHH(\AAA),\hHH(\AAA)$ are independent of the choice of $a$. By Remark \ref{R:bdrySob}, the spaces $\cHH(\AAA)$ and $\hHH(\AAA)$ are dual to each other.

The Sobolev spaces discussed above can be defined in the same way for the bundle $E^-$ using the restriction $\AAA^\sharp$ of $\DD^-$ to the boundary. It follows from \eqref{E:anticommu} that

\begin{lemma}\label{L:admap}
Over $\dm$, for all $s\in\RR$, the isomorphism $c(\tau):E^\pm|_{\dm}\to E^\mp|_{\dm}$ induces isomorphisms $H_{(-\infty,a]}^{s}(\AAA)\cong H_{[-a,\infty)}^{s}(\AAA^\sharp)$. In particular, $\cHH(\AAA)\cong\hHH(\AAA^\sharp)$, $\hHH(\AAA)\cong\cHH(\AAA^\sharp)$. 
\end{lemma}

\subsection{Boundary value problems}\label{SS:bvp}

One of the main results of \cite{BrShi17} is the regularity of maximal domain as below.

\begin{theorem}[\cite{BrShi17}*{\S3}]\label{T:maxdom}
Let $\DD^+$ be a strongly Callias-type operator. Then the trace map
\[
\begin{aligned}
\cRR\;:\;C_c^\infty(M,E^+)&\;\to\;C_c^\infty(\dm,E^+|_{\dm}) \\
u\quad&\;\mapsto\;\quad u|_{\dm}
\end{aligned}
\]
extends uniquely to a surjective bounded linear map $\cRR:\dom\DD_{\max}^+\to\cHH(\AAA)$.

The corresponding statement holds for $\DD_{\max}^-$ (with $\AAA$ replaced with $\AAA^\sharp$). Moreover, for all sections $u\in\dom\DD_{\max}^+$ and $v\in\dom\DD_{\max}^-$, we have the \emph{generalized Green's formula}
\begin{equation}\label{E:genGreensfor}
	\big(\DD_{\max}^+u;v\big)_{L^2(M)}
	\;-\;
	\big(u;\DD_{\max}^-v\big)_{L^2(M)}\;=\;
	-\,\big(c(\tau)\cRR u;\cRR v\big)_{L^2(\dm)}.
\end{equation}
\end{theorem}

This theorem inspires the following description of boundary value problems.

\begin{definition}[\cite{BrShi17}*{\S4}]\label{D:bc}
A closed subspace $B\subset\cHH(\AAA)$ is called a \emph{boundary condition} for $\DD^+$. We will use the notation $\DD_B^+$ for the operator with the domain
\[
	\dom(\DD_B^+)\;:=\;\{u\in\dom\DD_{\max}^+:\cRR u\in B\}.
\]
Its adjoint operator is $\DD_{B^\ad}^-$ with domain
\[
	\dom\DD_{B^\ad}^- \;= \;
	\big\{\,v\in\dom\DD_{\max}^-:
	\big(c(\tau)\cRR u;\cRR v\big)_{L^2(\dm)}=0\mbox{ for all }u\in\dom\DD_B^+\,\big\}.
\]
And
\begin{equation}\label{E:adbc}
	B^{\rm ad}\;:=\;
	\big\{\,\bfv\in\cHH(\AAA^\sharp):
		\big(c(\tau)\bfu;\bfv\big)_{L^2(\dm)}=0\mbox{ for all }\bfu\in B\,\big\}\;=\;\big(c(\tau)B\big)^0
\end{equation}
is called the \emph{adjoint boundary condition} of $B$, where the superscript ``0'' means the annihilator.
\end{definition}

\begin{definition}[\cite{BrShi17}*{\S4}]\label{D:ellbc}
A boundary condition $B$ is said to be \emph{elliptic} if $B\subset H_\AAA^{1/2}(\dm,E^+|_{\dm})$ and $B^{\rm ad}\subset H_{\AAA^\sharp}^{1/2}(\dm,E^-|_{\dm})$.
\end{definition}

\begin{example}[\cite{BrShi17}*{\S4}]\label{Ex:gAPS}
$B=H_{(-\infty,0)}^{1/2}(\AAA)$ is an elliptic boundary condition for $\DD^+$, which is called the \emph{Atiyah--Patodi--Singer boundary condition} (or APS boundary condition). Its adjoint boundary condition is $B^{\ad}=H_{(-\infty,0]}^{1/2}(\AAA^\sharp)$ and is called \emph{dual} APS boundary condition. In this case, we use notations
\[
\DD_{\rm APS}^+\;:=\;\DD_B^+,\qquad\DD_{\rm dAPS}^-\;:=\;\DD_{B^{\ad}}^-.
\]
\end{example}

A nice property of elliptic boundary value problems is the Fredholmness.

\begin{theorem}[\cite{BrShi17}*{\S5}]\label{T:Fred}
Let $\DD_B^+:\dom\DD_B^+\to L^2(M,E^-)$ be a strongly Callias-type operator with elliptic boundary condition. Then $\DD_B^+$ is a Fredholm operator.
\end{theorem}

In this case, the integer 
\begin{equation}\label{E:defofind}
		\ind\DD_B^+\;:=\;\dim\ker\DD_B^+-\dim\ker\DD_{B^{\rm ad}}^-\;\in\;\ZZ
\end{equation}
is called the {\em index of the boundary value problem} $\DD_B^+$.

\subsection{Unique continuation property}\label{SS:UCP}

We state a well-known property of Dirac-type operators, called the (weak) \emph{unique continuation property}, as follows

\begin{theorem}\label{T:UCP}
Let $P$ be a Dirac-type operator over a (connected) smooth manifold $M$. Then any smooth solution $s$ of $Ps=0$ which vanishes on an open subset of $M$ also vanishes on the whole manifold $M$.
\end{theorem}

Essentially, this property only depends on the symmetry of the principal symbol of Dirac-type operators and a nice proof is given in \cite{BW93}*{\S8}, \cite{Booss00}. In particular, the strongly Callias-type operators introduced earlier satisfy this property.

\begin{corollary}\label{C:UCP}
Let $\DD^+$ be a strongly Callias-type operator. Then the space of interior solutions
\[
\ker_0\DD_{\max}^+\;:=\;\{u\in\dom\DD_{\max}^+\;:\;\DD_{\max}^+u=0\mbox{ and }\cRR(u)=0\}
\]
contains only 0-sections. The same conclusion is true for $\DD^-$.
\end{corollary}

\begin{proof}
Proceeding as in \cite{BW93}*{\S9}, one can construct an invertible double $\tilde{\DD^+}$ of $\DD^+$ on $\tilde{M}$, the double of $M$, such that $
\tilde{\DD^+}|_M=\DD^+$. Let $u$ be an element of $\ker_0\DD_{\max}^+$. We extend it by zero to get a section $\tilde{u}$ on $\tilde{M}$. For any compactly supported smooth section $\tilde{v}$ on $\tilde{M}$, by Green's formula,
\[
\begin{aligned}
\big(\tilde{u};(\tilde{\DD^+})^*\tilde{v}\big)_{L^2(M)}&=\int_M\big\langle u;(\tilde{\DD^+})^*\tilde{v}|_M\big\rangle \\
&=\int_M\big\langle\DD^+u;\tilde{v}|_M\big\rangle+\int_{\p M}\langle c(\tau)u|_{\p M};\tilde{v}|_{\p M}\rangle=0.
\end{aligned}
\]
Thus $\tilde{u}$ is a weak solution of $\tilde{\DD^+}s=0$. By elliptic regularity, $\tilde{u}$ is smooth. Since $\tilde{u}$ vanishes on $\tilde{M}\setminus M$, applying Theorem \ref{T:UCP} to $\tilde{\DD^+}$ yields that $\tilde{u}\equiv0$ on $\tilde{M}$. Therefore $u$ is a 0-section.
\end{proof}

It follows from the corollary that

\begin{corollary}\label{C:UCP-2}
The maps $\cRR|_{\ker\DD_{\max}^\pm}:\ker\DD_{\max}^\pm\to\cHH(\AAA)$ (or $\cHH(\AAA^\sharp)$) are injective.
\end{corollary}

\begin{lemma}\label{L:UCP}
$\range\DD_{\max}^+=L^2(M,E^-).$
\end{lemma}

\begin{proof}
Since $\range\DD_{\max}^+\supset\range\DD_{\rm APS}^+$ and the latter admits a closed finite-dimensional complementary subspace in $L^2(M,E^-)$ (by the Fredholmness of $\DD_{\rm APS}^+$), one gets that $\range\DD_{\max}^+$ is closed in $L^2(M,E^-)$. Therefore
\[
\range\DD_{\max}^+\;=\;(\ker\DD_{\min}^-)^\perp\;=\;\{0\}^\perp\;=\;L^2(M,E^-).
\]
\end{proof}

\section{Maximal Cauchy data spaces and index formulas}\label{S:maxCauchy}

\begin{definition}\label{D:maxCauchy}
Let $\DD^+$ be a strongly Callias-type operator on $M$. We call 
\[
\CCC_{\max}^+\;:=\;\cRR(\ker\DD_{\max}^+)\;\subset\;\cHH(\AAA)
\]
the \emph{Cauchy data space} of the maximal extension $\DD_{\max}^+$. Similarly,
\[
\CCC_{\max}^-\;:=\;\cRR(\ker\DD_{\max}^-)\;\subset\cHH(\AAA^\sharp)
\]
is called the \emph{Cauchy data space} of the maximal extension $\DD_{\max}^-$.

Note that $\CCC_{\max}^+$ (resp. $\CCC_{\max}^-$) is a closed subspace of $\cHH(\AAA)$ (resp. $\cHH(\AAA^\sharp)$).
\end{definition}

\subsection{Fredholm pair}\label{SS:Fredpair}

We recall the concept of Fredholm pair (cf. \cite{Kato95}*{\S\upper{\romannumeral4}.4.1}).

\begin{definition}\label{D:Fredpair}
Let $Z$ be a Hilbert space. A pair $(X,Y)$ of closed subspaces of $Z$ is called a \emph{Fredholm pair} if
\begin{enumerate}[label=(\roman*)]
\item $\dim(X\cap Y)<\infty$;
\item $X+Y$ is a closed subspace of $Z$;
\item $\codim(X+Y):=\dim Z/(X+Y)<\infty$.
\end{enumerate}
The \emph{index} of a Fredholm pair $(X,Y)$ is defined to be
\[
\ind(X,Y)\;:=\;\dim(X\cap Y)-\codim(X+Y).
\]
\end{definition}

\begin{proposition}\label{P:Fredpair}
$(H_{(-\infty,0)}^{1/2}(\AAA),\CCC_{\max}^+)$ and $(H_{(-\infty,0]}^{1/2}(\AAA^\sharp),\CCC_{\max}^-)$ are Fredholm pairs in $\cHH(\AAA)$ and $\cHH(\AAA^\sharp)$, respectively. Moreover,
\begin{equation}\label{E:indeq}
\ind(H_{(-\infty,0)}^{1/2}(\AAA),\CCC_{\max}^+)\;=\;\ind\DD_{\rm APS}^+\;=\;-\ind(H_{(-\infty,0]}^{1/2}(\AAA^\sharp),\CCC_{\max}^-).
\end{equation}
\end{proposition}

The idea of the proof is from \cite{BF98}*{Proposition 3.5}.

\begin{proof}
Since $\ind\DD_{\rm APS}^+=-\ind\DD_{\rm dAPS}^-$ by \eqref{E:defofind}, we may only prove the conclusion for the first pair.

Recall that by Example \ref{Ex:gAPS}, $H_{(-\infty,0)}^{1/2}(\AAA)=\cRR(\dom\DD_{\rm APS}^+)$ and by Definition \ref{D:maxCauchy}$, \CCC_{\max}^+=\cRR(\ker\DD_{\max}^+)$. We first show that
\begin{equation}\label{E:intequ}
\cRR(\dom\DD_{\rm APS}^+\cap\ker\DD_{\max}^+)\;=\;\cRR(\dom\DD_{\rm APS}^+)\cap\cRR(\ker\DD_{\max}^+).
\end{equation}
It is clear that the right hand side includes the left hand side. To show the other direction, let $\bfu\in\cRR(\dom\DD_{\rm APS}^+)\cap\cRR(\ker\DD_{\max}^+)$. Then $\bfu=\cRR(u_1)=\cRR(u_2)$ for some $u_1\in\dom\DD_{\rm APS}^+$, $u_2\in\ker\DD_{\max}^+$. So $u_1-u_2\in\dom\DD_{\max}^+$ and $\cRR(u_1-u_2)=0$, which implies that $u_1-u_2\in\dom\DD_{\rm APS}^+$. Hence $u_2\in\dom\DD_{\rm APS}^+$ and it follows that $u_2\in\dom\DD_{\rm APS}^+\cap\ker\DD_{\max}^+$. Therefore $\bfu\in\cRR(\dom\DD_{\rm APS}^+\cap\ker\DD_{\max}^+)$. \eqref{E:intequ} is verified.

Since $\DD_{\rm APS}^+$ is a Fredholm operator, it follows from Corollary \ref{C:UCP-2} that
\begin{multline*}
\infty>\dim\ker\DD_{\rm APS}^+=\dim(\dom\DD_{\rm APS}^+\cap\ker\DD_{\max}^+) \\
=\dim\cRR(\dom\DD_{\rm APS}^+\cap\ker\DD_{\max}^+)=\dim(H_{(-\infty,0)}^{1/2}(\AAA)\cap\CCC_{\max}^+).
\end{multline*}
(\romannumeral1) of Definition \ref{D:Fredpair} is proved.

Note that the preimage of $\range\DD_{\rm APS}^+$ under $\DD^+$ is $\dom\DD_{\rm APS}^++\ker\DD_{\max}^+$. Since $\DD^+:\dom\DD_{\max}^+\to L^2(M,E^-)$ is continuous,
\[
\begin{aligned}
\DD_{\rm APS}^+\mbox{ Fredholm }&\;\Rightarrow\;\range\DD_{\rm APS}^+\mbox{ is closed in }L^2(M,E^-) \\
&\;\Rightarrow\;\dom\DD_{\rm APS}^++\ker\DD_{\max}^+\mbox{ is closed in }\dom\DD_{\max}^+.
\end{aligned}
\]
Recall that in \cite{BrShi17}, we defined a continuous extending map $\EE:\cHH(\AAA)\to\dom\DD_{\max}^+$ satisfying $\cRR\circ\EE=\id$. If $\{\bfu_j\}$ is a sequence in $\cRR(\dom\DD_{\rm APS}^++\ker\DD_{\max}^+)=H_{(-\infty,0)}^{1/2}(\AAA)+\CCC_{\max}^+\subset\cHH(\AAA)$ that is convergent to some $\bfu\in\cHH(\AAA)$, then $\{\EE\bfu_j\}$ converges to $\EE\bfu$ in $\dom\DD_{\max}^+$. Like what we argued in proving \eqref{E:intequ}, using the fact that $\dom\DD_{\rm APS}^++\ker\DD_{\max}^+$ is a subspace of $\dom\DD_{\max}^+$, one can show that $\EE\bfu_j\in\dom\DD_{\rm APS}^++\ker\DD_{\max}^+$. By the above closedness, $\EE\bfu$ also lies in $\dom\DD_{\rm APS}^++\ker\DD_{\max}^+$. Therefore $\bfu=\cRR(\EE\bfu)\in H_{(-\infty,0)}^{1/2}(\AAA)+\CCC_{\max}^+$. (\romannumeral2) of Definition \ref{D:Fredpair} is proved.

To prove Definition \ref{D:Fredpair}.(\romannumeral3) and equation \eqref{E:indeq}, note that $\cRR$ induces a bijection between $\dom\DD_{\max}^+/(\dom\DD_{\rm APS}^++\ker\DD_{\max}^+)$ and $\cHH(\AAA)/(H_{(-\infty,0)}^{1/2}(\AAA)+\CCC_{\max}^+)$. Let $\pi:L^2(M,E^-)\twoheadrightarrow(\range\DD_{\rm APS}^+)^\perp$ be the orthogonal projection. By Lemma \ref{L:UCP}, $\DD_{\max}^+:\dom\DD_{\max}^+\to L^2(M,E^-)$ is surjective, so
\[
\ker(\pi\circ\DD_{\max}^+)\;=\;\dom\DD_{\rm APS}^++\ker\DD_{\max}^+.
\]
Then
\[
\begin{aligned}
\dom\DD_{\max}^+/(\dom\DD_{\rm APS}^++\ker\DD_{\max}^+)\;&\cong\;(\range\DD_{\rm APS}^+)^\perp \\
&=\;L^2(M,E^-)/\range\DD_{\rm APS}^+.
\end{aligned}
\]
Hence
\begin{equation}\label{E:leftineq}
\begin{aligned}
\codim(H_{(-\infty,0)}^{1/2}(\AAA)+\CCC_{\max}^+)&\;=\;\dim\cHH(\AAA)/(H_{(-\infty,0)}^{1/2}(\AAA)+\CCC_{\max}^+) \\
&\;=\;\dim\dom\DD_{\max}^+/(\dom\DD_{\rm APS}^++\ker\DD_{\max}^+) \\
&\;=\;\dim L^2(M,E^-)/\range\DD_{\rm APS}^+ \\
&\;=\;\dim\coker\DD_{\rm APS}^+\;<\;\infty.
\end{aligned}
\end{equation}
Therefore
\[
\ind(H_{(-\infty,0)}^{1/2}(\AAA),\CCC_{\max}^+)\;=\;\ind\DD_{\rm APS}^+.
\]
\end{proof}

\subsection{Fredholm pair of projections}\label{SS:Fredproj}

A notion that is closely related to Fredholm pair is the \emph{Fredholm pair of projections} considered in \cite{ASS}.

\begin{definition}\label{D:Fredproj}
Let $Z$ be a Hilbert space and $(X,Y)$ be a pair of closed subspaces of $Z$. Denote the orthogonal projections from $Z$ onto $X,Y$ by $P_X,P_Y$, respectively. $(P_X,P_Y)$ is called a \emph{Fredholm pair of projections} if $P_XP_Y:\range P_Y\to\range P_X$ is a Fredholm operator. Its index is defined as $\ind(P_X,P_Y):=\ind P_XP_Y$.
\end{definition}

We formulate the following standard result about equivalent definitions of Fredholm pairs and Fredholm pair of projections (cf. \cite{Kato95}*{\S\upper{\romannumeral4}.4.2}, \cite{BW93}*{\S24}).

\begin{proposition}\label{P:Fredequiv}
Let $Z$ be a Hilbert space and $X,Y,P_X,P_Y$ be as above. Then the following are equivalent:
\begin{enumerate}
\item $(X,Y)$ is a Fredholm pair;
\item $(X^0,Y^0)$ is a Fredholm pair, where $X^0,Y^0\subset Z^*$ are the annihilators of $X,Y$, respectively;
\item $(X^\perp,Y^\perp)$ is a Fredholm pair, where $X^\perp,Y^\perp\subset Z$ are the orthogonal complements of $X,Y$, respectively;
\item $(P_{X^\perp},P_Y)$ is a Fredholm pair of projections.
\end{enumerate}
In this case, one has
\[
\begin{aligned}
\dim(X\cap Y)\;=\;\codim(X^0+Y^0)\;=\;\codim(X^\perp+Y^\perp)\;=\;\dim\ker P_{X^\perp}P_Y; \\
\codim(X+Y)\;=\;\dim(X^0\cap Y^0)\;=\;\dim(X^\perp\cap Y^\perp)\;=\;\codim\range P_{X^\perp}P_Y.
\end{aligned}
\]
In particular,
\[
\ind(X,Y)\;=\;-\ind(X^0,Y^0)\;=\;-\ind(X^\perp,Y^\perp)\;=\;\ind(P_{X^\perp},P_Y).
\]
\end{proposition}

We return to Cauchy data spaces. Let $\cPi_+(\AAA)$ be the orthogonal projection $\cHH(\AAA)\twoheadrightarrow H_{[0,\infty)}^{-1/2}(\AAA)$ and $\cPP(\DD^+)$ be the orthogonal projection $\cHH(\AAA)\twoheadrightarrow\CCC_{\max}^+$. Let $\cTT:=\cPi_+(\AAA)\cPP(\DD^+):\CCC_{\max}^+\to H_{[0,\infty)}^{-1/2}(\AAA)$. The following is a quick consequence of Propositions \ref{P:Fredpair} and \ref{P:Fredequiv}.

\begin{theorem}\label{T:indid}
$\cTT$ is a Fredholm operator and $\ind\cTT=\ind\DD_{\rm APS}^+$.
\end{theorem}

\subsection{$L^2$-situation}\label{SS:L2situ}

We define the \emph{$L^2$-Cauchy data space} $\CCC^+:=\CCC_{\max}^+\cap L^2(\p M,E^+|_{\p M})$. One can apply the idea of ``criss-cross reduction'' in \cite{BFO} to show that $\CCC^+$ is a closed subspace of $L^2(\p M,E^+|_{\p M})$. We briefly present this argument. First, there exists a closed subspace $V\subset\cHH(\AAA)$, such that $\CCC_{\max}^+$ can be written as a direct sum of transversal (not necessarily orthogonal) pair of subspaces
\[
\CCC_{\max}^+\;=\;(H_{(-\infty,0)}^{1/2}(\AAA)\cap\CCC_{\max}^+)\;\dot{+}\;V.
\]
Let $\pi_+$ (resp. $\pi_-$) be the projection of $V$ onto $H_{[0,\infty)}^{-1/2}(\AAA)$ (resp. $H_{(-\infty,0)}^{1/2}(\AAA)$) along $H_{(-\infty,0)}^{1/2}(\AAA)$ (resp. $H_{[0,\infty)}^{-1/2}(\AAA)$). Then $\pi_+$ is injective and $\range\pi_+=\range\cTT$ is closed. By closed graph theorem, $\pi_+$ has a bounded inverse $\iota_+:\range\pi_+\to V$. We then have a bounded operator $\check{\phi}:=\pi_-\circ\iota_+:\range\pi_+\to\range\pi_-$. This gives another expression of $\CCC_{\max}^+$:
\begin{equation}\label{E:Cmaxdecomp}
\CCC_{\max}^+\;=\;(H_{(-\infty,0)}^{1/2}(\AAA)\cap\CCC_{\max}^+)+{\rm graph}(\check{\phi}).
\end{equation}

Let $\phi$ be the restriction of $\check{\phi}$ to $L^2(\p M,E^+|_{\p M})$. Then $\dom\phi$ is closed in $L^2(\p M,E^+|_{\p M})$. Viewed as an operator $\dom\phi\to L^2(\p M,E^+|_{\p M})$, $\phi$ is still bounded. Note that now $\CCC^+$ can be written as
\[
\CCC^+\;=\;(H_{(-\infty,0)}^{1/2}(\AAA)\cap\CCC_{\max}^+)+{\rm graph}(\phi).
\]
Since the first summand is finite-dimensional, $\CCC^+$ is closed in $L^2(\p M,E^+|_{\p M})$.

Like in Subsection \ref{SS:Fredproj}, we define the orthogonal projections
\[
\Pi_+:L^2(\p M,E^+|_{\p M})\twoheadrightarrow L_{[0,\infty)}^2(\AAA)\quad\mbox{and}\quad P(\DD^+):L^2(\p M,E^+|_{\p M})\twoheadrightarrow\CCC^+.
\]
And let
\[
T:=\Pi_+(\AAA)P(\DD^+):\CCC^+\to L_{[0,\infty)}^2(\AAA).
\]
It is clear that $\ker T=\ker\cTT$, and
\[
\begin{aligned}
\range T&\;=\;(L_{(-\infty,0)}^2(\AAA)+\CCC^+)\cap L_{[0,\infty)}^2(\AAA) \\
&\;=\;(L_{(-\infty,0)}^2(\AAA)+\CCC_{\max}^+)\cap L_{[0,\infty)}^2(\AAA) \\
&\;\supset\;(H_{(-\infty,0)}^{1/2}(\AAA)+\CCC_{\max}^+)\cap L_{[0,\infty)}^2(\AAA).
\end{aligned}
\]
On the other hand, since the $L^2$-norm is stronger than the $\cHH$-norm on $L_{[0,\infty)}^2(\AAA)$,
\[
\begin{aligned}
\range T&\;=\;({\rm cl}_{L^2}(H_{(-\infty,0)}^{1/2}(\AAA))+\CCC_{\max}^+)\cap L_{[0,\infty)}^2(\AAA) \\
&\;\subset\;({\rm cl}_{\cHH}(H_{(-\infty,0)}^{1/2}(\AAA))+\CCC_{\max}^+)\cap L_{[0,\infty)}^2(\AAA) \\
&\;\subset\;{\rm cl}_{\cHH}(H_{(-\infty,0)}^{1/2}(\AAA)+\CCC_{\max}^+)\cap L_{[0,\infty)}^2(\AAA) \\
&\;=\;(H_{(-\infty,0)}^{1/2}(\AAA)+\CCC_{\max}^+)\cap L_{[0,\infty)}^2(\AAA),
\end{aligned}
\]
where we used Proposition \ref{P:Fredpair} in the last line. Therefore
\[
\begin{aligned}
\range T&\;=\;(H_{(-\infty,0)}^{1/2}(\AAA)+\CCC_{\max}^+)\cap L_{[0,\infty)}^2(\AAA) \\
&\;=\;\range\cTT\cap L_{[0,\infty)}^2(\AAA).
\end{aligned}
\]
and is a closed subspace of $L_{[0,\infty)}^2(\AAA)$. Let $\check{W}$ be the finite-dimensional orthogonal complement of $\range\cTT$ in $H_{[0,\infty)}^{-1/2}(\AAA)$ and let $W:=\check{W}|_{L_{[0,\infty)}^2(\AAA)}$. Then
\begin{equation}\label{E:decomp}
L_{[0,\infty)}^2(\AAA)\;=\;\range T+W.
\end{equation}
Taking closure with respect to the $\cHH$-norm for both sides implies that $H_{[0,\infty)}^{-1/2}(\AAA)=\range\cTT+W$. Hence $W=\check{W}$. It follows that \eqref{E:decomp} is a direct sum decomposition. Therefore
\[
\codim\range T\;=\;\dim W\;=\;\dim\check{W}\;=\;\codim\range\cTT.
\]

To sum up, we obtain an $L^2$-version of Theorem \ref{T:indid}:

\begin{theorem}\label{T:L2indid}
$T$ is a Fredholm operator and $\ind T=\ind\DD_{\rm APS}^+$.
\end{theorem}

\begin{corollary}\label{C:L2indid}
$(L_{(-\infty,0)}^2(\AAA),\CCC^+)$ is a Fredholm pair in $L^2(\p M,E^+|_{\p M})$ and
\[
\ind(L_{(-\infty,0)}^2(\AAA),\CCC^+)\;=\;\ind\DD_{\rm APS}^+.
\]
\end{corollary}

\section{Cauchy data spaces and boundary value problems}\label{S:Cauchy-bvp}

\subsection{Twisted orthogonality of Cauchy data spaces}\label{SS:orth}

By Proposition \ref{P:Fredequiv} and Corollary \ref{C:L2indid}, $(L_{[0,\infty)}^2(\AAA),(\CCC^+)^0)$ and $(L_{(0,\infty)}^2(\AAA^\sharp),(\CCC^-)^0)$ are Fredholm pairs in $L^2(\p M,E^+|_{\p M})$ and $L^2(\p M,E^-|_{\p M})$, respectively. And they satisfy
\begin{multline}\label{E:annihindid}
\ind(L_{[0,\infty)}^2(\AAA),(\CCC^+)^0)\;=\;-\ind(L_{(-\infty,0)}^2(\AAA),\CCC^+), \\
\;=\;\ind(L_{(-\infty,0]}^2(\AAA^\sharp),\CCC^-)\;=\;-\ind(L_{(0,\infty)}^2(\AAA^\sharp),(\CCC^-)^0).
\end{multline}

The following property of Fredholm pairs can be verified easily.

\begin{lemma}\label{L:inclu-equal}
Let $(X,Y_1),(X,Y_2)$ be two Fredholm pairs in a Hilbert space $Z$. If $Y_1\subset Y_2$ and $\ind(X,Y_1)=\ind(X,Y_2)$, then $Y_1=Y_2$.
\end{lemma}

\begin{proposition}\label{P:anniheq}
Recall that $c(\tau)$ induces an isomorphism between $L^2(\p M,E^-|_{\p M})$ and $L^2(\p M,E^+|_{\p M})$. Then $c(\tau)(\CCC^-)=(\CCC^+)^0$, $c(\tau)(\CCC^+)=(\CCC^-)^0$.
\end{proposition}

\begin{proof}
We only need to show the first equality. Let $\bfv\in\CCC^-$. Then there exists a $v\in\ker\DD_{\max}^+$ such that $\cRR(v)=\bfv$. For any $\bfu\in\CCC^+$, there again exists a $u\in\ker\DD_{\max}^+$ such that $\cRR(u)=\bfu$. By \eqref{E:genGreensfor},
\[
0=(\DD_{\max}^+u;v)_{L^2(M)}-(u;\DD_{\max}^-v)_{L^2(M)}=(\bfu,c(\tau)\bfv)_{L^2(\p M)}\ \Rightarrow\ c(\tau)\bfv\in(\CCC^+)^0.
\]
Hence $c(\tau)(\CCC^-)\subset(\CCC^+)^0$.

Notice that the isomorphism $c(\tau)$ maps the Fredholm pair $(L_{(-\infty,0]}^2(\AAA^\sharp),\CCC^-)$ to the pair $(L_{[0,\infty)}^{1/2}(\AAA),c(\tau)(\CCC^-))$. Thus the latter is a Fredholm pair in $L^2(\p M,E^+|_{\p M})$ and 
\[
\ind(L_{[0,\infty)}^2(\AAA),c(\tau)(\CCC^-))\;=\;\ind(L_{(-\infty,0]}^2(\AAA^\sharp),\CCC^-)\;\xlongequal{\eqref{E:annihindid}}\;\ind(L_{[0,\infty)}^{1/2}(\AAA),(\CCC^+)^0).
\]
Using the fact that $c(\tau)(\CCC^-)\subset(\CCC^+)^0$ and Lemma \ref{L:inclu-equal}, one has $c(\tau)(\CCC^-)=(\CCC^+)^0$.
\end{proof}

\begin{remark}\label{R:anniheq}
In the same way, one can prove that $c(\tau)(\CCC_{\max}^-)=(\CCC_{\max}^+)^0$, $c(\tau)(\CCC_{\max}^+)=(\CCC_{\max}^-)^0$.
\end{remark}

Since the pairing between elements of $(L^2(\p M,E^+|_{\p M}))^*\cong L^2(\p M,E^+|_{\p M})$ and elements of $L^2(\p M,E^+|_{\p M})$ coincides with the inner product on $L^2(\p M,E^+|_{\p M})$, we have $(\CCC^+)^\perp=(\CCC^+)^0$. 
Therefore, we obtian the following $L^2$-decomposition theorem.

\begin{theorem}\label{T:L2decomp}
$\CCC^+$ and $c(\tau)(\CCC^-)$ are orthogonal complementary subspaces of $L^2(\p M,E^+|_{\p M})$. Similar statement is true for $\CCC^-$ and $c(\tau)(\CCC^+)$.
\end{theorem}

Consider a bilinear form on $L^2(\p M,E|_{\p M})$ defined by
\[
\omega(\bfu,\bfv)\;:=\;(c(\tau)\bfu;\bfv)_{L^2(\p M)}.
\]
One can check that this is a symplectic form. Then Theorem \ref{T:L2decomp} indicates the following.

\begin{corollary}\label{C:Lagrangian}
The total $L^2$-Cauchy data space $\CCC^+\oplus\CCC^-$ of the total strongly Callias-type operator $\DD$ is a Lagrangian subspace of $L^2(\p M,E|_{\p M})$.
\end{corollary}

\begin{remark}\label{R:Lagrangian}
From Remark \ref{R:anniheq}, one can also show that the total maximal Cauchy data spaces $\CCC_{\max}^+\oplus\CCC_{\max}^-$ is a Lagrangian subspace of $\cHH(\AAA)\oplus\cHH(\AAA^\sharp)$.
\end{remark}

\subsection{Cauchy data spaces as elliptic boundary conditions}\label{SS:Cauchy-bc}

In this subsection, we discuss an elliptic boundary condition induced by Cauchy data spaces.

Let
\[
\CCC_{1/2}^+\;:=\;\CCC_{\max}^+\cap H^{1/2}_\AAA(\p M,E^+|_{\p M}),\quad
\CCC_{1/2}^-\;:=\;\CCC_{\max}^-\cap H^{1/2}_{\AAA^\sharp}(\p M,E^-|_{\p M}).
\]
Using again the expression \eqref{E:Cmaxdecomp} of $\CCC_{\max}^+$, like in Subsection \ref{SS:L2situ}, we have
\begin{equation}\label{E:C1/2decomp}
\CCC_{1/2}^+\;=\;(H_{(-\infty,0)}^{1/2}(\AAA)\cap\CCC_{\max}^+)+{\rm graph}(\phi^{1/2}),
\end{equation}
where $\phi^{1/2}:\dom\phi^{1/2}\to H_{(-\infty,0)}^{-1/2}(\AAA)$ is the restriction of $\check{\phi}$ to $H_{[0,\infty)}^{1/2}(\AAA)$, and it is still a bounded operator. So $\CCC_{1/2}^+$ is a closed subspace of $\hHH(\AAA)$, and $c(\tau)(\CCC_{1/2}^+)$ is a closed subspace of $\cHH(\AAA^\sharp)$. Similarly, $c(\tau)(\CCC_{1/2}^-)$ is a closed subspace of $\cHH(\AAA)$.

\begin{lemma}\label{L:C1/2Fredpair}
$(H_{(-\infty,0)}^{-1/2}(\AAA),\CCC_{1/2}^+)$ is a Fredholm pair in $\hHH(\AAA)$ and
\[
\ind(H_{(-\infty,0)}^{-1/2}(\AAA),\CCC_{1/2}^+)\;=\;\ind(H_{(-\infty,0)}^{1/2}(\AAA),\CCC_{\max}^+).
\]
\end{lemma}

\begin{proof}
First,
\[
\begin{aligned}
H_{(-\infty,0)}^{-1/2}(\AAA)\cap\CCC_{1/2}^+\;&=\;H_{(-\infty,0)}^{-1/2}(\AAA)\cap\CCC_{\max}^+\cap H^{1/2}_\AAA(\p M,E^+|_{\p M}) \\
&=\;H_{(-\infty,0)}^{1/2}(\AAA)\cap\CCC_{\max}^+.
\end{aligned}
\]
By \eqref{E:C1/2decomp},
\[
H_{(-\infty,0)}^{-1/2}(\AAA)+\CCC_{1/2}^+\;=\;H_{(-\infty,0)}^{-1/2}(\AAA)+{\rm graph}(\phi^{1/2})\;=\;H_{(-\infty,0)}^{-1/2}(\AAA)\oplus\dom\phi^{1/2},
\]
which is closed in $\hHH(\AAA)$. Then
\begin{multline*}
\dim\hHH(\AAA)/(H_{(-\infty,0)}^{-1/2}(\AAA)+\CCC_{1/2}^+)\;=\;\dim H_{[0,\infty)}^{1/2}(\AAA)/\dom\phi^{1/2} \\
\;=\;\dim H_{[0,\infty)}^{-1/2}(\AAA)/\dom\phi\;=\;\dim\cHH(\AAA)/(H_{(-\infty,0)}^{1/2}(\AAA)+\CCC_{\max}^+).
\end{multline*}
The lemma is proved.
\end{proof}

\begin{remark}\label{R:C1/2Fredpair}
One also has that $(H_{(-\infty,0]}^{-1/2}(\AAA^\sharp),\CCC_{1/2}^-)$ is a Fredholm pair in $\hHH(\AAA^\sharp)$ and
\[
\ind(H_{(-\infty,0]}^{-1/2}(\AAA^\sharp),\CCC_{1/2}^-)\;=\;\ind(H_{(-\infty,0]}^{1/2}(\AAA^\sharp),\CCC_{\max}^-).
\]
\end{remark}

\begin{theorem}\label{T:Cauchyellbc}
$c(\tau)(\CCC_{1/2}^-)$ is an elliptic boundary condition for $\DD^+$, whose adjoint boundary condition is $c(\tau)(\CCC_{1/2}^+)$ and $\ind\DD_{c(\tau)(\CCC_{1/2}^-)}^+=0$.
\end{theorem}

\begin{proof}
From the discussion above, $c(\tau)(\CCC_{1/2}^-)\subset H_\AAA^{1/2}(\p M,E^+|_{\p M})$ and is a boundary condition. By \eqref{E:adbc}, to prove the adjoint property, it suffices to show that $c(\tau)(\CCC_{1/2}^+)=(\CCC_{1/2}^-)^0$.

Note that $c(\tau)$ maps the Fredholm pair $(H_{(-\infty,0)}^{-1/2}(\AAA),\CCC_{1/2}^+)$ of $\hHH(\AAA)$ to a Fredholm pair $(H_{(0,\infty)}^{-1/2}(\AAA^\sharp),c(\tau)(\CCC_{1/2}^+))$ of $\cHH(\AAA^\sharp)$ and
\begin{multline*}
\ind(H_{(0,\infty)}^{-1/2}(\AAA^\sharp),c(\tau)(\CCC_{1/2}^+))=\ind(H_{(-\infty,0)}^{1/2}(\AAA),\CCC_{\max}^+) \\
\xlongequal{c(\tau)}\ind(H_{(0,\infty)}^{1/2}(\AAA^\sharp),c(\tau)(\CCC_{\max}^+))\xlongequal{\text{Remark }\ref{R:anniheq}}\ind(H_{(0,\infty)}^{1/2}(\AAA^\sharp),(\CCC_{\max}^-)^0) \\
=-\ind(H_{(-\infty,0]}^{1/2}(\AAA^\sharp),\CCC_{\max}^-)\xlongequal{\text{Remark }\ref{R:C1/2Fredpair}}-\ind(H_{(-\infty,0]}^{-1/2}(\AAA^\sharp),\CCC_{1/2}^-) \\
=\ind(H_{(0,\infty)}^{-1/2}(\AAA^\sharp),(\CCC_{1/2}^-)^0).
\end{multline*}
One then uses the argument as in the proof of Proposition \ref{P:anniheq} to show that $c(\tau)(\CCC_{1/2}^+)\subset(\CCC_{1/2}^-)^0$. Therefore $c(\tau)(\CCC_{1/2}^+)=(\CCC_{1/2}^-)^0$ by Lemma \ref{L:inclu-equal}.

By Theorem \ref{T:L2decomp}, one gets
\[
\begin{aligned}
c(\tau)(\CCC_{1/2}^-)\;\subset\;c(\tau)(\CCC^-)\;=\;(\CCC^+)^\perp\;\Longrightarrow\;&c(\tau)(\CCC_{1/2}^-)\cap\CCC^+ \\
\;=\;&c(\tau)(\CCC_{1/2}^-)\cap\CCC_{\max}^+\cap L^2(\p M,E^+|_{\p M}) \\
\;=\;&c(\tau)(\CCC_{1/2}^-)\cap\CCC_{\max}^+\;=\;\{0\}.
\end{aligned}
\]
So $\ker\DD_{c(\tau)(\CCC_{1/2}^-)}^+=\{0\}$. Also $\ker\DD_{c(\tau)(\CCC_{1/2}^+)}^-=\{0\}$. Hence
\[
\ind\DD_{c(\tau)(\CCC_{1/2}^-)}^+\;=\;\dim\ker\DD_{c(\tau)(\CCC_{1/2}^-)}^+-\dim\ker\DD_{c(\tau)(\CCC_{1/2}^+)}^-\;=\;0.
\]
\end{proof}

\section*{Acknowledgments}
The author would like to thank Professor Maxim Braverman for all the helpful discussions and suggestions in writing this paper.

\bibliographystyle{amsplain}

\end{document}